\documentclass[12pt]{article}

\usepackage{dsfont}
\usepackage{amsfonts,amsmath,amsthm}
\usepackage{algorithm, algorithmic}
\usepackage{color}
\usepackage{graphicx}
\usepackage{stmaryrd}        % provides \llbracket and \rrbracket

\usepackage[paper=a4paper,dvips,top=2cm,left=2cm,right=2cm,
    foot=1cm,bottom=4cm]{geometry}

\begin{document}

\title{A Low Rank Quaternion Decomposition Algorithm and Its Application in Color Image Inpainting}
\author{Yannan Chen\footnote{%
    School of Mathematical Sciences, South China Normal University, Guangzhou, China; ({\tt ynchen@scnu.edu.cn}).
    This author was supported by the National Natural Science Foundation of China (11771405) and
    the Natural Science Foundation of Guangdong Province, China (2020A1515010489).}
  \and
    Liqun Qi\footnote{%
    Department of Applied Mathematics, The Hong Kong Polytechnic University,
    Hung Hom, Kowloon, Hong Kong ({\tt maqilq@polyu.edu.hk}).}
    \and
    Xinzhen Zhang\thanks{School of Mathematics, Tianjin University, Tianjin 300354 China; ({\tt xzzhang@tju.edu.cn}). This author's work was supported by NSFC (Grant No.  11871369). }
  \and and \
    Yanwei Xu\thanks{Huawei Theory Research Lab, Hong Kong, China; ({\tt xuyanwei1@huawei.com}).}
    \thanks{Future Network Theory Lab, 2012 Labs
Huawei Tech. Investment Co., Ltd, Shatin, New Territory, Hong Kong, China; ({\tt xuyanwei1@huawei.com}).}
}
\date{\today}
\maketitle

\begin{abstract}
In this paper, we propose a lower rank quaternion decomposition algorithm and apply it to color image inpainting.   We introduce a concise form for the gradient of a real function in quaternion matrix variables.    The optimality conditions of our quaternion least squares problem have a simple expression with this form.  The convergence and convergence rate of our algorithm are established with this tool.  %Numerical tests on practical datasets demonstrate the efficiency of the algorithm.

  \medskip


  %\medskip
  %\textbf{AMS subject classifications.}
\end{abstract}

\renewcommand{\Re}{\mathds{R}}
\newcommand{\rank}{\mathrm{rank}}
\renewcommand{\span}{\mathrm{span}}
\newcommand{\X}{\mathcal{X}}
\newcommand{\A}{\mathcal{A}}
\newcommand{\B}{\mathcal{B}}
\newcommand{\C}{\mathcal{C}}
\newcommand{\OO}{\mathcal{O}}
\newcommand{\e}{\mathbf{e}}
\newcommand{\0}{\mathbf{0}}
\newcommand{\dd}{\mathbf{d}}
\newcommand{\ii}{\mathbf{i}}
\newcommand{\jj}{\mathbf{j}}
\newcommand{\kk}{\mathbf{k}}
\newcommand{\va}{\mathbf{a}}
\newcommand{\vb}{\mathbf{b}}
\newcommand{\vc}{\mathbf{c}}
\newcommand{\vg}{\mathbf{g}}
\newcommand{\vr}{\mathbf{r}}
\newcommand{\vt}{\rm{vec}}
\newcommand{\vx}{\mathbf{x}}
\newcommand{\vy}{\mathbf{y}}
\newcommand{\y}{\mathbf{y}}
\newcommand{\vz}{\mathbf{z}}
\newcommand{\T}{\top}

\newtheorem{Thm}{Theorem}[section]
\newtheorem{Def}[Thm]{Definition}
\newtheorem{Ass}[Thm]{Assumption}
\newtheorem{Lem}[Thm]{Lemma}
\newtheorem{Prop}[Thm]{Proposition}
\newtheorem{Cor}[Thm]{Corollary}

\section{introduction}

In 1996, Sangwine \cite{Sa96} proposed to encode three channel components of an RGB (Red-Green-Blue) image on the three imaginary parts of a pure quaternion.   In 1997, Zhang \cite{Zh97} established the quaternion singular value decomposition (QSVD) of a quaternion matrix.   Also see \cite{LS03, SL06}.  Since then, quaternion matrices and their QSVD are used widely in color image processing, including color image denoising and inpainting \cite{JNW19, XZ18, XYXZN15, ZKW16}.   Recently, Chen, Xiao and Zhou \cite{CXZ19} proposed low rank quaternion approximation (LRQA) for color image denoising and inpainting.   They proposed to use the quaternion rank or some rank surrogates of a quaternion variable matrix $X$ as a regularized term in the approximation of $X$ to a given quaternion data matrix $Y$.

In this paper, we propose a lower rank quaternion decomposition (LRQD) algorithm and apply it to color image inpainting.
To conduct convergence analysis for our algorithm, we need to consider optimization problems of real functions in quaternion matrix variables.   To handle this, we introduce a concise form for the gradient of a real function in quaternion matrix variables.   The optimality conditions of our quaternion least squares problem have a simple expression with this form.
With this tool, convergence and convergence rate of our algorithm are established.  %Numerical tests on practical datasets demonstrate the efficiency of the algorithm.

The rest of the paper is organized as follows.   Some necessary preliminary knowledge on quaternions and quaternion matrices are given in the next section.  A lower rank quaternion decomposition algorithm (LRQD)  is presented in Section 3.   In Section 4, we introduce a concise form for the gradient of a real function in quaternion matrix variables.   The gradients of some real functions  in quaternion matrix variables, which are useful for our least squares model, have simple expressions.  The convergence analysis of our algorithm is presented in Section 5.  We show that the {K}urdyka-{\L}ojasiewicz inequality \cite{AB09, ABRS10, BDL07} holds for the critical points of our algorithm  there.  The LRQD algorithm is applied to color image inpainting.  %Numerical tests on practical datasets are reported in Section 6.
In Section 6, we make some concluding remarks.

\section{Preliminary}

\subsection{Quaternions}

In this paper, the real field, the complex field and the quaternion field are denoted by ${\mathbb R}$, $\mathbb C$ and $\mathbb Q$, respectively.  Furthermore, scalars, vectors, matrices and tensors are denoted by small letters, bold small letters, capital letters and calligraphic letters, respectively.   We denote vectors with matrix components by bold capital letters.   For example, we denote ${\bf Z} = (A, B, X)$.
We use $\0, O$ and $\OO$ to denote zero vector, zero matrix and zero tensor with adequate dimensions.
An exception is that $\ii, \jj$ and $\kk$ denote the three imaginary units of quaternions.  We use the notation in \cite{Zh97, WLZZ18}.    We have
$$\ii^2 = \jj^2 = \kk^2 =\ii\jj\kk = -1,$$
$$\ii\jj = -\jj\ii = \kk, \ \jj\kk = - \kk\jj = \ii, \kk\ii = -\ii\kk = \jj$$.
These rules, along with the distribution law, determine the product of two quaternions.   Hence, the multiplication of quaternions is noncommutative.

Let $x = x_0 + x_1\ii + x_2\jj + x_3\kk \in {\mathbb Q}$, where $x_0, x_1, x_2, x_3 \in {\mathbb R}$.
Then the conjugate of $x$ is
$$\bar x \equiv x^* = x_0 - x_1\ii - x_2\jj - x_3\kk,$$
the modulus of $x$ is
$$|x| = |x^*| = \sqrt{xx^*} = \sqrt{x^*x} = \sqrt{x_0^2 + x_1^2 + x_2^2 + x_3^2},$$
and if $x \not = 0$, then $x^{-1} = {x^* \over |x|^2}$.

\subsection{Quaternion Matrices}

Denote the collections of real, complex and quaternion $m \times n$ matrices by ${\mathbb R}^{m \times n}$, ${\mathbb C}^{m \times n}$ and ${\mathbb Q}^{m \times n}$, respectively.

Then $A= (a_{ij}) \in {\mathbb Q}^{m \times n}$ can be denoted as
\begin{equation} \label{e1}
A = A_0 + A_1\ii + A_2\jj + A_3\kk,
\end{equation}
where $A_0, A_1, A_2, A_3 \in {\mathbb R}^{m \times n}$.   The transpose of $A$ is $A^\T = (a_{ji})$. The conjugate of $A$ is $\bar A = (a_{ij}^*)$.   The conjugate transpose of $A$ is $A^* = (a_{ji}^*) = \bar A^T$.   The Frobenius norm of $A$ is
$$\|A\|_F = \sqrt{\sum_{i=1}^m \sum_{j=1}^n |a_{ij}|^2}.$$
By \cite{WLZZ18}, we have the following proposition.

\begin{Prop} \label{p2.1}
Suppose that $A \in {\mathbb Q}^{m \times r}$ and $B \in {\mathbb Q}^{m \times r}$.  Then
$$\|AB\|_F \le \|A\|_F \|B \|_F.$$
\end{Prop}

A square matrix $A \in {\mathbb Q}^{m \times m}$ is called a unitary matrix if and only if $AA^* = A^*A = I_m$, where $I_m \in {\mathbb R}^{m \times m}$ is the real $m \times m$ identity matrix.

For the complex representation of a quaternion matrix $A \in {\mathbb Q}^{m \times n}$ with the expression (\ref{e1}), we follow \cite{WLZZ18}, and denote it as $A^C$.  Let $B_1 = A_0 + A_1\ii$ and $B_2 = A_3 + A_4\ii$.   Then $B_1, B_2 \in {\mathbb C}^{m \times n}$, and $A = B_1 + B_2\jj$.  The complex representation of $A$ is
$$A^C = \left( {\ B_1 \ B_2 \atop -\bar B_2 \ \bar B_1} \right).$$

A color image dataset can be expressed as a third order tensor $\A \in {\mathbb R}^{m \times n \times 3}$.
On the other hand, we may also represent it by a pure quaternion matrix
$$A = A(:,:,1)\ii + A(:,:,2)\jj + A(:,:,3)\kk \in {\mathbb Q}^{m \times n},$$
where  $A(:,:,1), A(:,:,2)$ and $A(:,:,3)$ are the three frontal slices of $\A$.

We have the following theorem for the QSVD of a quaternion matrix \cite{Zh97}.

\begin{Thm} {\bf (Zhang 1997)} \label{t2.1}
Any quaternion matrix $X \in {\mathbb Q}^{m \times n}$ \cite{Zh97} has the following QSVD form
\begin{equation} \label{e2}
X = U\left({\Sigma_r \ O \atop O \ \ O}\right)V^*,
\end{equation}
where $U \in {\mathbb Q}^{m \times m}$ and $V \in {\mathbb Q}^{n \times n}$ are unitary,  and $\Sigma_r$ = diag$\{ \sigma_1, \cdots, \sigma_r\}$ is a real nonnegative $r \times r$ diagonal matrix, with $\sigma_1 \ge \cdots \ge \sigma_r$ as the singular values of $X$.
\end{Thm}

The quaternion rank of $X$ is the number of its positive singular values, denoted as rank$(X)$.

\begin{Cor} \label{C2.2}
Suppose that quaternion matrix $X \in {\mathbb Q}^{m \times n}$.   Then
\begin{equation} \label{e3}
{\rm rank}(X) \le \min \{ m, n \}.
\end{equation}
\end{Cor}

For the ranks of quaternion matrices, we have the following theorem.    This theorem can be found on Page 295 of \cite{Xi82} and Page 35 of \cite{Zh06}.   Since these two references are not in English, we give a proof here for completeness.

\begin{Thm} \label{t2.3}
Suppose that $A \in {\mathbb Q}^{m \times r}$ and $B \in {\mathbb Q}^{r \times n}$.
Then
$${\rm rank}(AB) \le \min \{ {\rm rank}(A), {\rm rank}(B) \}.$$
In particular,
\begin{equation} \label{e4}
{\rm rank}(AB) \le r.
\end{equation}
\end{Thm}
\begin{proof}
(1) We first show that ${\rm rank}(AB)\leq {\rm rank}(B)$.

For any $x\in  {\mathbb Q}^n$ satisfying  $Bx=0$, we have  $ABx=0$. This means that  $\mathcal{N}(B)\subseteq \mathcal{N}(AB)$, where $\mathcal{N}(C)$ denotes the null space of matrix $C$. Hence, ${\rm dim} \mathcal{N}(B)\leq {\rm dim} \mathcal{N}(AB)$ and  ${\rm rank}(AB)\leq {\rm rank}(B)$
since ${\rm rank}(C)=n-{\rm dim}\mathcal{N}(C)$.

 To show that ${\rm rank}(AB)\leq {\rm rank}(A)$. Let $A_1,A_2,\cdots, A_m$ be all rows of matrix $A$. Then  $A_1B, A_2B,\cdots, A_mB$ are all rows of matrix $AB$.
Let ${\rm rank}(AB)=r$. Without loss of generality, we assume that $A_1B,\dots, A_rB$ are left linearly independent. Now we show that $A_1,\dots, A_r$ are linearly independent by contradiction.

Assume that
there exist $k_1,
\dots, k_r\in \mathbb{Q}$ such that $k_1A_1+\dots+k_rA_r=0$. Then \[k_1A_1B+\dots+k_rA_rB=(k_1A_1+\dots+k_rA_r)B=0.\] From $A_1B,\dots, A_rB$ are left linearly independent \cite{WLZZ18}, we have
$k_1=\dots=k_r=0$ and hence $A_1,\dots, A_r$ are left linearly independent. This means that ${\rm rank}(A)\geq r$ since ${\rm rank}(A)$ is the maximum number of rows that are left linearly independent.   That is, ${\rm rank}(A)\geq {\rm rank}(AB)$.

In all, ${\rm rank}(AB)\leq {\rm min} \{{\rm rank}(A), {\rm rank}(B)\}$.

From the fact that ${\rm rank}(A)\leq r$, ${\rm rank}(AB)\leq r$ is clear.

\end{proof}

We use $O_{m \times n}$ to denote the zero matrix in ${\mathbb Q}^{m \times n}$.

\section{A LRQD Algorithm}

Suppose that we have a quaternion data matrix $D \in {\mathbb Q}^{m \times n}$, which is only partially observed.  Let $\Omega$ be the set of observed entries of $D$.
The low rank quaternion approximation (LRQA) model for color image inpaiting is as follows:
\begin{equation} \label{e5}
\min_{X \in {\mathbb Q}^{m \times n}} \left\{ {1 \over 2} \|(X-D)_{\Omega}\|_F^2 : {\rm rank}(X) = r \right\},
\end{equation}
where $r < \min \{ m, n \}$.

We derive a low rank quaternion decomposition theorem for a quaternion matrix $X$.

\begin{Thm}  \label{t3.1}
Suppose that $X \in {\mathbb Q}^{m \times n}$.  Let positive integer $r < \{ m, n \}$.   Then rank$(X) \le r$ if and only if there are  $A \in {\mathbb Q}^{m \times r}$ and $B \in {\mathbb Q}^{r \times n}$
such that $X = AB$.
\end{Thm}
\begin{proof}    If $X = AB$, $A \in {\mathbb Q}^{m \times r}$ and $B \in {\mathbb Q}^{r \times n}$, then by Theorem \ref{t2.3}, we have rank$(X) \le r$.

On the other hand, suppose that rank$(X) \le r$.   Then by Theorem \ref{t2.1}, we have (\ref{e2}).  Let
$$U =(U_1 \ U_2)$$
and
$$V^* = \left({V_1 \atop V_2}\right)$$
where $U_1 \in {\mathbb Q}^{m \times r}$, $U_2 \in {\mathbb Q}^{m \times (m-r)}$, $V_1 \in {\mathbb Q}^{r \times n}$ and $V_2 \in {\mathbb Q}^{(n-r) \times n}$.  Let $\Sigma$ = diag$\{ \sqrt{\sigma_1}, \cdots, \sqrt{\sigma_r}\}$, $A= U_1\Sigma$ and $B = \Sigma V_1$.   Then we have $X = AB$, $A \in {\mathbb Q}^{m \times r}$ and $B \in {\mathbb Q}^{r \times n}$.
\end{proof}

We now propose a LRQA model
\begin{equation} \label{e6}
\min_{A \in {\mathbb Q}^{m \times r}, B \in {\mathbb Q}^{r \times n}, X \in {\mathbb Q}^{m \times n}}  \left\{f(A, B, X) \equiv {1 \over 2} \|AB-X\|_F^2 : X_{\Omega} = D_{\Omega} \right\}.
\end{equation}
Here, $r < \min \{ m, n \}$, $X \in {\mathbb Q}^{m \times n}$ is a surrogate matrix.  We may use the following alternative scheme to find $A, B$ and $X$.

First, we have $A^{(0)} \in {\mathbb Q}^{m \times r}$ and $B^{(0)} \in {\mathbb Q}^{r \times n}$.

At the $k$th iteration, we first calculate $X^{(k)}$.  Let $X^{(k)}$ be the solution of
\begin{equation} \label{e6a}
\min_{X \in {\mathbb Q}^{m \times n}}  \left\{p(X) \equiv {1 \over 2} \|A^{(k)}B^{(k)}-X\|_F^2 : X_{\Omega} = D_{\Omega} \right\}.
\end{equation}
Then,
\begin{equation} \label{e7}
X^{(k)}_{\Omega} = D_{\Omega}, \ X^{(k)}_{\Omega_C} =\left(A^{(k)}B^{(k)}\right)_{\Omega_C},
\end{equation}
where $\Omega_C$ is the complement of $\Omega$.

We find $A^{(k+1)}$ as the solution of
\begin{equation} \label{e8a}
\min_{A \in {\mathbb Q}^{m \times r}} g(A) \equiv {1 \over 2}  \left\|AB^{(k)}-X^{(k)}\right\|_F^2 +  {\lambda \over 2}\left\|A - A^{(k)}\right\|^2_F,
\end{equation}
where $\lambda > 0$ is a regularization coefficient.

\begin{Prop} \label{pp3.2}
The least squares problem (\ref{e8a}) has an explicit solution
\begin{equation} \label{e8}
A^{(k+1)} = \left[X^{(k)}\left(B^{(k)}\right)^*+ \lambda A^{(k)}\right]\left[B^{(k)}\left(B^{(k)}\right)^*+ \lambda I_r\right]^{-1}.
\end{equation}
\end{Prop}

 We will prove this proposition in the next section.

We find $B^{(k+1)}$ as the solution of
\begin{equation} \label{e9a}
\min_{B \in {\mathbb Q}^{r \times n}} h(B) \equiv {1 \over 2} \left\|A^{(k+1)}B-X^{(k)}\right\|_F^2 + {\lambda \over 2}\left\|B-B^{(k)}\right\|_F^2,
\end{equation}
where $\lambda > 0$.
Similarly, we have the following proposition.

\begin{Prop} \label{pp3.3}
The least squares problem (\ref{e9a}) has an explicit solution
\begin{equation} \label{e9}
B^{(k+1)} = \left[\left(A^{(k+1)}\right)^*A^{(k+1)}+ \lambda I_r\right]^{-1}\left[\left(A^{(k+1)}\right)^*X^{(k)} + \lambda B^{(k)}\right].
\end{equation}
\end{Prop}

We thus have the following algorithm to solve (\ref{e6}).

%Denote the row vectors of $A$ by $\va_1^\T, \cdots, \va_m^\T$, and the row vectors of $Y$ by $\vy_1^\T, \cdots, \vy_m^\T$.   Then we may decompose (\ref{e7}) to $m$ vector optimization problems
%\begin{equation} \label{e9}
%\min_{\va_i \in {\mathbb Q}^r} {1 \over 2} \left\|\va_i^\T B^{(k)}- \vy_i^\T %\right\|_F^2,
%\end{equation}
%for $i = 1, \cdots, m$.

%Denote the column vectors of $B$ by $\vb_1, \cdots, \vb_n$, and the column vectors of $Y$ by $\vz_1, \cdots, \vz_n$.   Then we may decompose (\ref{e8}) to $n$ vector optimization problems
%\begin{equation} \label{e10}
%\min_{\vb_j \in {\mathbb Q}^r} {1 \over 2} %\left\|A^{(k+1)}\vb_j-\vz_j\right\|_F^2.
%\end{equation}
%for $j = 1, \cdots, n$.

%In practice, $m$ and $n$ are big, but $r$ is small.   Hence, we may solve $m+n$ small-scale optimization problems (\ref{e9}) and (\ref{e10}) in a parallel way by the least squares method at each iteration.  Note that (\ref{e9}) and (\ref{e10}) can be solved directly.

\begin{algorithm} \label{LRQD}
\caption{A LRQD algorithm for solving (\ref{e6}).}\label{lrqd}
\begin{algorithmic}[1]
  \STATE Given $\epsilon > 0$, $A^{(0)} \in {\mathbb Q}^{m \times r}$ and $B^{(0)} \in {\mathbb Q}^{r \times n}$.  Set $k\gets 0$.
  %\WHILE{$\left\|(A^{(k)}B^{(k)}-D)_{\Omega}\right\|_F^2 > \epsilon$}
      \STATE Use (\ref{e7}) to find $X^{(k)}$.
      \STATE Use (\ref{e8}) to find $A^{(k+1)}$.
      \STATE Use (\ref{e9}) to find $B^{(k+1)}$.
      \STATE If $A^{(k+1)}= A^{(k)}$ and $B^{(k+1)} = B^{(k)}$, stop.  Otherwise, $k\gets k+1$ and goto Step 2.
  %\ENDWHILE
\end{algorithmic}
\end{algorithm}

\section{The Gradient of A Real Function in Quaternion Matrix Variables}

To conduct convergence analysis for our algorithm, we need to consider optimization problems of real functions in quaternion matrix variables.   To handle this, we introduce a concise form for the gradient of a real function in quaternion matrix variables.

Consider the following optimization problem
\begin{equation} \label{n13}
\min \{ f(X) : X \in {\mathbb Q}^{m \times n}, g_k(X) = 0, k = 1, \cdots, p \},
\end{equation}
where $f, g_k: {\mathbb Q}^{m \times n} \to {\mathbb R}$ for $k = 1, \cdots, p$.    We need to have a form for $\nabla f(X)$ and $\nabla g_k(X)$ for $k = 1, \cdots, p$.

\begin{Def} \label{d4.1}
Let $f : {\mathbb Q}^{m \times n} \to {\mathbb R}$.   Let $X = X_0 + X_1\ii +X_2\jj + X_3\kk$, where $X_0, X_1, X_2, X_3 \in {\mathbb R}^{m \times n}$.  Then we say that $f$ is differentiable at $X$ if ${\partial f \over \partial X_i}$ exists at $X_i$ for $i=0, 1, 2, 3$, and we denote
\begin{equation} \label{n14}
\nabla f(X) = {\partial f \over \partial X_0} + {\partial f \over \partial X_1}\ii + {\partial f \over \partial X_2}\jj + {\partial f \over \partial X_3}\kk.
\end{equation}
If ${\partial f \over \partial X_i}$ exists in a neighborhood of $X_i$, and is continuous at $X_i$, for $i=0, 1, 2, 3$, then we say that $f$ is continuously differentiable at $X$.    If $f$ is continuously differentiable for any $X \in {\mathbb Q}^{m \times n}$, then we say that $f$ is continuously differentiable.
\end{Def}

If $f$ have more variables, then we may change $\nabla f(X)$ in (\ref{n14}) to $\partial f \over \partial X$.

This form is different from the generalized HR calculus studied in \cite{XJTM15}, which is similar to the approach in optimization of real functions with complex variables \cite{SBD12}.

Based from this definition, we have the following theorem.

\begin{Thm} \label{nt4.2}
Suppose that $f, g_k: {\mathbb Q}^{m \times n} \to {\mathbb R}$ for $k = 1, \cdots, p$, are continuously differentiable, and $X^{\#} \in {\mathbb Q}^{m \times n}$ is an optimal solution of (\ref{n13}).   Then there are Langrangian multipliers $\lambda_k \in {\mathbb R}$ for $k = 1, \cdots, p$, such that
$$\nabla f(X^{\#}) + \sum_{j=1}^p \lambda_k \nabla g_k(X^{\#}) = O.$$
\end{Thm}
\begin{proof}
Let $X = X_0 + X_1\ii +X_2\jj + X_3\kk$, where $X_0, X_1, X_2, X_3 \in {\mathbb R}^{m \times n}$.  Then (\ref{n13}) is converted to an optimization problem with real matrix variables.  By optimization theory and Definition \ref{d4.1}, the conclusion holds.
\end{proof}

This theorem can be extended to other optimization problems involving continuously differentiable real functions in quaternion matrix variables.   The Langrangian multipliers $\lambda_k$ are real numbers.

With Definition \ref{d4.1}, the gradients of some functions useful for our model have simple expressions.

\begin{Thm} \label{t2.4}
Suppose that $f : {\mathbb Q}^{m \times r}$ be defined by $f(X) = {1 \over 2}\|XB + C\|_F^2$, where $B \in {\mathbb Q}^{r \times n}$ and $C \in {\mathbb Q}^{m \times n}$.   Then
$$\nabla f(X) = (XB + C)B^*.$$
\end{Thm}
\begin{proof}  We have $XB + C= M_0+ M_1\ii + M_2\jj+M_3\kk$,
where
$$M_0 = X_0B_0-X_1B_1-X_2B_2-X_3B_3+C_0,$$
$$M_1 = X_0B_1+X_1B_0+X_2B_3-X_3B_2+C_1,$$
$$M_2 = X_0B_2+X_2B_0+X_1B_3-X_3B_1+C_2,$$
$$M_3 = X_0B_3+X_3B_0+X_1B_2-X_2B_1+C_3.$$
Then,
$$f(X) = {1 \over 2}\sum_{i=0}^3 \|M_i\|_F^2,$$
$${\partial f \over \partial X_0} = M_0B_0^\T + M_1B_1^\T + M_2B_2^\T + M_3B_3^\T,$$
$${\partial f \over \partial X_1} = -M_0B_1^\T + M_1B_0^\T + M_2B_3^\T + M_3B_2^\T,$$
$${\partial f \over \partial X_2} = -M_0B_2^\T + M_1B_3^\T + M_2B_0^\T - M_3B_1^\T,$$
$${\partial f \over \partial X_3} = -M_0B_3^\T - M_1B_2^\T - M_2B_1^\T + M_3B_0^\T.$$
Thus, we have
$$\nabla f(X) = {\partial f \over \partial X_0} + {\partial f \over \partial X_1}\ii + {\partial f \over \partial X_2}\jj + {\partial f \over \partial X_3}\kk = (XB+C)(B_0-B_1^\T\ii-B_2^\T\jj-B_3^\T\kk) = (XB+C)B^*.$$
\end{proof}

We now prove Proposition \ref{pp3.2}.   For
$$g(A) \equiv {1 \over 2}  \left\|AB^{(k)}-X^{(k)}\right\|_F^2 +  {\lambda \over 2}\left\|A - A^{(k)}\right\|^2_F,$$ with a similar argument as the proof of Theorem \ref{t2.4},
we have
$$\nabla g(A) = \left(AB^{(k)} -X^{(k)}\right)\left(B^{(k)}\right)^* + \lambda \left(A - A^{(k)}\right).$$
Then the optimality condition of (\ref{e8a}) is
$$\nabla g(A) = O_{m \times r}.$$
We have
$$AB^{(k)}\left(B^{(k)}\right)^* + \lambda A = X^{(k)}\left(B^{(k)}\right)^*  + \lambda A^{(k)},$$
i.e.,
$$A\left[B^{(k)}\left(B^{(k)}\right)^* + \lambda I_r\right] = X^{(k)}\left(B^{(k)}\right)^*  + \lambda A^{(k)},$$
Since $\lambda > 0$, the inverse of $B^{(k)}\left(B^{(k)}\right)^* + \lambda I_r$ exists.
From this, we see that (\ref{e8}) is the solution of (\ref{e8a}).

Similarly, we have the following theorem.

\begin{Thm} \label{t2.5}
Suppose that $f : {\mathbb Q}^{r \times n}$ be defined by $f(X) = {1 \over 2}\|AX + C\|_F^2$, where $A \in {\mathbb Q}^{m \times r}$ and $C \in {\mathbb Q}^{m \times n}$.   Then
$$\nabla f(X) = A^*(AX + C).$$
\end{Thm}

We also can prove Proposition \ref{pp3.3} similarly.

In our convergence analysis, the {K}urdyka-{\L}ojasiewicz property \cite{AB09, ABRS10, BDL07} plays a critical role.   Since we regard functions in this paper as functions defined on an abstract vector space with real coefficients, the {K}urdyka-{\L}ojasiewicz property also holds for functions related with the optimization problems studied in this paper.

\section{Convergence Analysis}

In this section, we present convergence analysis for the LRQD algorithm.  As stated in Subsection 2.3, we may regard the objective function $f$ in (\ref{e6}) as a function defined in the abstract vector space of dimension $4(mr+rn+mn)$ with real coefficients.   Then we use the gradient of $f$, and study the stationary points and the first order optimality conditions of (\ref{e6})
without ambiguity.

In the following, we always denote ${\bf Z} \equiv (A, B, X)$.   Thus, ${\bf Z}^{(k)} \equiv (A^{(k)}, B^{(k)}, X^{(k)})$, ${\bf Z}^{\#} \equiv (A^{\#}, B^{\#}, X^{\#})$, so on.

We have the following theorem.

\begin{Thm}
Let $A^{\#} \in {\mathbb Q}^{m \times r}, B^{\#} \in {\mathbb Q}^{r \times n}$ and $X^{\#} \in {\mathbb Q}^{m \times n}$.   Suppose that $X^{\#}$ satisfies
\begin{equation} \label{e10}
X^{\#}_{\Omega} = D_{\Omega},
\end{equation}
i.e., $X^{\#}$ is a feasible point of (\ref{e6}).   If ${\bf Z}^{\#} \equiv \left(A^{\#}, B^{\#}, X^{\#}\right)$ is an optimal solution of (\ref{e6}), then we have
\begin{equation} \label{e11}
\left(A^{\#}B^{\#}-X^{\#}\right)\left(B^{\#}\right)^{*} = O_{m \times r},
\end{equation}
\begin{equation} \label{e12}
\left(A^{\#}\right)^{*}\left(A^{\#}B^{\#}-X^{\#}\right) = O_{r \times n},
\end{equation}
\begin{equation} \label{e13}
X^{\#}_{\Omega_C} = (A^{\#}B^{\#})_{\Omega_C},
\end{equation}
i.e., ${\bf Z}^{\#} \equiv \left(A^{\#}, B^{\#}, X^{\#}\right)$ is a stationary point of (\ref{e6}).
\end{Thm}
\begin{proof}  By Theorem \ref{nt4.2}, ${\bf Z}^{\#} \equiv \left(A^{\#}, B^{\#}, X^{\#}\right)$ should satisfy
$$\begin{array}{rl}
{\partial \over \partial A}f(A^{\#}, B^{\#}, X^{\#}) & = O_{m \times r}, \\
{\partial \over \partial B}f(A^{\#}, B^{\#}, X^{\#}) & = O_{r \times n}, \\
\left({\partial \over \partial X}f(A^{\#}, B^{\#}, X^{\#})\right)_{\Omega_C} & = \left(O_{m \times n}\right)_{\Omega_C} \\
\left({\partial \over \partial X}f(A^{\#}, B^{\#}, X^{\#})\right)_{ij} & = -\pi_{ij}, \ \ {\rm for}\ (i, j) \in \Omega,
\end{array}$$
where $\pi_{ij}$ for $(i, j) \in \Omega$ are Langrangian multipliers.

By Theorem \ref{t2.4},
$${\partial \over \partial A}f(A^{\#}, B^{\#}, X^{\#}) = \left(A^{\#}B^{\#}-X^{\#}\right)\left(B^{\#}\right)^{*}.$$
We have (\ref{e11}).

By Theorem \ref{t2.5},
$${\partial \over \partial B}f(A^{\#}, B^{\#}, X^{\#}) =
\left(A^{\#}\right)^{*}\left(A^{\#}B^{\#}-X^{\#}\right).$$
We have (\ref{e12}).

By Theorems \ref{t2.4} or \ref{t2.5},
$${\partial \over \partial X}f(A^{\#}, B^{\#}, X^{\#}) =
X^{\#}- A^{\#}B^{\#}.$$
We have (\ref{e13}) and
$$\pi(ij) = \left(A^{\#}B^{\#}-X^{\#}\right)_{ij}$$
for $(i, j) \in \Omega$.   In (\ref{e6}), $X_\Omega = D_\Omega$ is a quaternion matrix equality.   If we convert it to real equalities, there are $4|\Omega|$ real equalities, i.e., (\ref{e6}) can be regarded an optimization problem in the form (\ref{e13}) with $4|\Omega|$ real equality constraints.  Then $\pi(i, j)_l$ for $l = 0, 1, 2, 3$, i.e., the coefficients of the quaternion number $\pi(i, j)$, are corresponding real Langrangian multipliers in Theorem \ref{nt4.2}.   This implies that ${\bf Z}^{\#} \equiv \left(A^{\#}, B^{\#}, X^{\#}\right)$ is a stationary point of (\ref{e6}), in the sense of Theorem \ref{nt4.2}.

The theorem is proved.
\end{proof}

We now consider the case that the LRQD algorithm stops in a finite number of iterations.

\begin{Prop} \label{p4.2}
If Algorithm 1 stops in a finite number of iterations, then ${\bf Z}^{\#} \equiv \left(A^{\#}, B^{\#}, X^{\#}\right) \equiv \left(A^{(k)}, B^{(k)}, X^{(k)}\right)$ is a stationary point of (\ref{e6}).
\end{Prop}
\begin{proof}
If Algorithm 1 stops in a finite number of iterations, then $A^{(k+1)} = A^{(k)}$ and $B^{(k+1)} = B^{(k)}$.   Denote $A^{\#} = A^{(k)}$, $B^{\#} = B^{(k)}$ and $X^{\#} = X^{(k)}$.   By (\ref{e7}), we have (\ref{e10}) and (\ref{e13}).  By (\ref{e8}), we have
$$A^{\#}\left(B^{\#}(B^{\#})^*+\lambda I_r\right) = X^{\#}(B^{\#})^* + \lambda A^{\#}.$$
This leads to (\ref{e12}).  By (\ref{e9}), we have
$$\left((A^{\#})(A^{\#}+\lambda I_r\right)B^{\#} = (A^{\#})^* X^{\#}+ \lambda B^{\#}.$$
This leads to (\ref{e11}).   These imply that ${\bf Z}^{\#} \equiv \left(A^{\#}, B^{\#}, X^{\#}\right) \equiv \left(A^{(k)}, B^{(k)}, X^{(k)}\right)$ is a stationary point of (\ref{e6}).
\end{proof}

We then consider the case that the LRQD algorithm generates an infinite sequence \\
$\left\{ {\bf Z}^{(k)} \equiv \left(A^{(k)}, B^{(k)}, X^{(k)}\right) : k = 0, 1, 2, \cdots \right\}$.

\begin{Lem} \label{l4.3}
Suppose that $\left\{ {\bf Z}^{(k)} \equiv \left(A^{(k)}, B^{(k)}, X^{(k)} \right): k = 0, 1, 2, \cdots \right\}$ is the sequence generated by the LRQD algorithm.   Then
$$\begin{array}{rl}
& f\left(A^{(k)}, B^{(k)}, X^{(k))}\right) - f\left(A^{(k+1)}, B^{(k+1)}, X^{(k+1)}\right) \\
\ge & {\lambda_0 \over 2} \left[\|A^{(k)}-A^{(k+1)}\|_F^2 + \|B^{(k)}-B^{(k+1)}\|_F^2 + \|X^{(k)}-X^{(k+1)}\|_F^2\right],
\end{array}$$
where $\lambda_0 = \min \{ \lambda, 1 \}$.
\end{Lem}
\begin{proof}  First, we have
$$\begin{array}{rl}
& f\left(A^{(k)}, B^{(k)}, X^{(k)}\right) - f\left(A^{(k+1)}, B^{(k)}, X^{(k)}\right) \\
= & {1 \over 2}\left\|A^{(k)}B^{(k)}-X^{(k)}\right\|_F^2 - {1 \over 2}\left\|A^{(k+1)}B^{(k)}-X^{(k)}\right\|_F^2 \\
= & g(A^{(k)}) - g(A^{(k+1)}) + {\lambda \over 2}\left\|A^{(k+1)}-A^{(k)}\right\|_F^2\\
\ge & {\lambda \over 2}\left\|A^{(k+1)}-A^{(k)}\right\|_F^2.
\end{array}$$
Here, $g$ is the objective function of (\ref{e8a}).   Since $A^{(k+1)}$ minimizes (\ref{e8a}), we have
$g(A^{(k)}) \ge g(A^{(k+1})$.  This leads to the last inequality.

Similarly, we have
$$\begin{array}{rl}
& f\left(A^{(k+1)}, B^{(k)}, X^{(k)}\right) - f\left(A^{(k+1)}, B^{(k+1)}, X^{(k)}\right) \\
= & {1 \over 2}\left\|A^{(k+1)}B^{(k)}-X^{(k)}\right\|_F^2 - {1 \over 2}\left\|A^{(k+1)}B^{(k+1)}-X^{(k)}\right\|_F^2 \\
= & h(B^{(k)}) - h(B^{(k+1)}) + {\lambda \over 2}\left\|B^{(k+1)}-B^{(k)}\right\|_F^2\\
\ge & {\lambda \over 2}\left\|B^{(k+1)}-B^{(k)}\right\|_F^2.
\end{array}$$
Here, $h$ is the objective function of (\ref{e9a}).   Since $B^{(k+1)}$ minimizes (\ref{e9a}), we have
$h(B^{(k)}) \ge h(B^{(k+1})$.  This leads to the last inequality.

Finally, we have
$$\begin{array}{rl}
& f\left(A^{(k+1)}, B^{(k+1)}, X^{(k)}\right) - f\left(A^{(k+1)}, B^{(k+1)}, X^{(k+1)}\right) \\
= & {1 \over 2}\left\|A^{(k+1)}B^{(k+1)}-X^{(k)}\right\|_F^2 - {1 \over 2}\left\|A^{(k+1)}B^{(k+1)}-X^{(k+1)}\right\|_F^2 \\
= & {1\over 2}\left\|(A^{(k+1)}B^{(k+1)}-X^{(k)})_\Omega\right\|_F^2+\left\|(A^{(k+1)}B^{(k+1)}-X^{(k)})_{\Omega_C}\right\|_F^2 \\&- {1 \over 2}\left\|(A^{(k+1)}B^{(k+1)}-X^{(k+1)})_\Omega\right\|_F^2 - {1 \over 2}\left\|(A^{(k+1)}B^{(k+1)}-X^{(k+1)})_{\Omega_C}\right\|_F^2 \\
= & {1 \over 2}\left\|(A^{(k+1)}B^{(k+1)}-X^{(k)})_{\Omega_C}\right\|_F^2\\
=& {1 \over 2}\left\|(X^{(k+1)}-X^{(k)})_{\Omega_C}\right\|_F^2\\
=& {1\over 2} \|X^{(k)}-X^{(k+1)}\|_F^2.
\end{array}$$
Summing up these three parts, we have the conclusion.
\end{proof}

\begin{Lem} \label{l4.4}
Suppose that $\left\{ {\bf Z}^{(k)} \equiv \left(A^{(k)}, B^{(k)}, X^{(k)} \right): k = 0, 1, 2, \cdots \right\}$ is the infinite sequence generated by the LRQD algorithm.   Then
$$\sum_{k=0}^\infty \left[\left\|A^{(k)} - A^{(k+1)}\right\|_F^2 + \left\|B^{(k)} - B^{(k+1)}\right\|_F^2 + \left\|X^{(k)} - X^{(k+1)}\right\|_F^2\right] < \infty,$$
$$\lim_{k \to \infty} \left[A^{(k)} - A^{(k+1)}\right] = O_{m,r},$$
$$\lim_{k \to \infty} \left[B^{(k)} - B^{(k+1)}\right] = O_{r,n}$$
and
$$\lim_{k \to \infty} \left[X^{(k)} - X^{(k+1)}\right] = O_{m,n}.$$
\end{Lem}
\begin{proof} By Lemma \ref{l4.3}, for $k = 0, 1, 2, \cdots$, we have
$$\begin{array}{rl}
& \left\|A^{(k)} - A^{(k+1)}\right\|_F^2 + \left\|B^{(k)} - B^{(k+1)}\right\|_F^2 + \left\|X^{(k)} - X^{(k+1)}\right\|_F^2 \\
\le & {2 \over \lambda}\left[f\left(A^{(k)}, B^{(k)}, X^{(k)}\right) - f\left(A^{(k+1)}, B^{(k+1)}, X^{(k+1)}\right)\right].
\end{array}$$
Summarizing with respect to $k$, we have
$$\begin{array}{rl}
& \sum_{k=0}^\infty \left[\left\|A^{(k)} - A^{(k+1)}\right\|_F^2 + \left\|B^{(k)} - B^{(k+1)}\right\|_F^2 + \left\|X^{(k)} - X^{(k+1)}\right\|_F^2\right] \\
\le & \sum_{k=0}^\infty {2 \over \lambda_0}\left[f\left(A^{(k)}, B^{(k)}, X^{(k)}\right) - f\left(A^{(k+1)}, B^{(k+1)}, X^{(k+1)}\right)\right] \\
\le & {2 \over \lambda_0}f\left(A^{(0)}, B^{(0)}, X^{(0))}\right) < \infty.
\end{array}$$
Hence, $\left\|A^{(k)} - A^{(k+1)}\right\|_F^2 \to 0$, $\left\|B^{(k)} - B^{(k+1)}\right\|_F^2 \to 0$ and
$\left\|X^{(k)} - X^{(k+1)}\right\|_F^2 \to 0$.   The second conclusion of the lemma follows.
\end{proof}

\begin{Thm}
Suppose that $\left\{ {\bf Z}^{(k)} \equiv \left(A^{(k)}, B^{(k)}, X^{(k)} \right) : k = 0, 1, 2, \cdots \right\}$ is the sequence generated by the LRQD algorithm and it is bounded.   Then every limiting point of this sequence is a stationary point of (\ref{e6}).
\end{Thm}
\begin{proof}  Since $\left\{ \left(A^{(k)}, B^{(k)}, X^{(k)} \right) : k = 0, 1, 2, \cdots \right\}$ is bounded, it must have a subsequence \\
$\left\{ \left(A^{(k_i)}, B^{(k_i)}, X^{(k_i)}\right) : i = 0, 1, 2, \cdots \right\}$ that converges to a limiting point $\left(A^{\#}, B^{\#}, X^{\#}\right)$.   By Lemma \ref{l4.4}, the subsequence
$\left\{ \left(A^{(k_i+1)}, B^{(k_i+1)}, X^{(k_i+1)}\right) : i = 0, 1, 2, \cdots \right\}$ converges to the same limiting point.

In (\ref{e7}), (\ref{e8}) and (\ref{e9}), replace $k$ by $k_i$ and let $i \to \infty$.   Then, with an argument similar to the proof of Proposition \ref{p4.2}, we conclude that $\left(A^{\#}, B^{\#}, X^{\#}\right)$ is a stationary point of (\ref{e6}).
\end{proof}

Consider optimization problem (\ref{e6}).   As we regard the objective function $f$ as a function defined on an abstract vector space with real coefficients, it is a semi-algebraic function in the sense of \cite{AB09, ABRS10, BDL07}.    Then, for any critical point $\left(A^{\#}, B^{\#}, X^{\#}\right)$ of $f$, there are a neighborhood $N$ of this critical point, an exponent $\theta \in [{1 \over 2}, 1)$ and a positive constant $\mu$ such that the {K}urdyka-{\L}ojasiewicz inequality \cite{AB09, ABRS10, BDL07} below holds, i.e., for any $(A, B, X) \in N$, we have
\begin{equation} \label{e16}
\left| f(A, B, X) - f\left(A^{\#}, B^{\#}, X^{\#}\right)\right|^\theta \le \mu \left\| \Pi_{\Sigma}(\nabla f(A, B, X))\right\|_F,
\end{equation}
where $\Sigma$ is the feasible set of (\ref{e6}), $\Pi_{\Sigma}(\nabla f)$ is the projected gradient of $f$ with respect to $\Sigma$.   Then, we have
\begin{equation} \label{e17}
\Pi_{\Sigma}(\nabla f(A, B, X)) = \left(\begin{array}{l} (AB-X)B^* \\ A^*(AB-X) \\ (X-AB)_{\Omega_C} \end{array}\right).
\end{equation}

We may further confirm (\ref{e16}) by the following argument in the real field.   Let
$$\begin{array}{rl}
A & = A_0 + A_1\ii + A_2\jj + A_3\kk, \\
B & = B_0 + B_1\ii + B_2\jj + B_3\kk, \\
X & = X_0 + X_1\ii + X_2\jj + X_3\kk, \\
D & = D_0 + D_1\ii + D_2\jj + D_3\kk, \\
{\bf W} & = (A_0, A_1, A_2, A_3, B_0, B_1, B_2, B_3, X_0, X_1, X_2, X_3),
\end{array}$$
where $A_i \in {\mathbb R}^{m \times r}, B_i \in {\mathbb R}^{r \times n}$, $X_i \in {\mathbb R}^{m \times n}$ and $D_i \in {\mathbb R}^{m \times n}$ for $i = 0, 1, 2, 3$.   Define $\phi({\bf W}) \equiv f(A, B, X)$.   Then
$$\begin{array}{rl} \phi({\bf W})  = & {1 \over 2} \left\|(A_0B_0 -A_1B_1-A_2B_2-A_3B_3-X_0) + (A_0B_1 +A_1B_0+A_2B_3-A_3B_2-X_1)\ii \right. \\
 & \left. + (A_0B_2 +A_2B_0+A_3B_1-A_1B_3-X_2)\jj + (A_0B_3 +A_3B_0+A_1B_2-A_2B_1-X_3)\kk \right\|_F^2\\
 = & {1 \over 2} \left[ \|A_0B_0 -A_1B_1-A_2B_2-A_3B_3-X_0\|_F^2 + \|A_0B_1 +A_1B_0+A_2B_3-A_3B_2-X_1\|_F^2\right. \\
 & \left. + \|A_0B_2 +A_2B_0+A_3B_1-A_1B_3-X_2\|_F^2 + \|A_0B_3 +A_3B_0+A_1B_2-A_2B_1-X_3\|_F^2\right].
\end{array}$$
This shows that $\phi$ is a semi-algebraic function.   The constraints of (\ref{e6}) in the real field are
$(X_i)_{\Omega} = (D_i)_{\Omega}$ for $i = 0, 1, 2, 3$.   Applying the {K}urdyka-{\L}ojasiewicz inequality in the real field \cite{AB09, ABRS10, BDL07}, we may still obtain (\ref{e16}).

In the following, we use (\ref{e16}) to show that the infinite sequence generated by the LRQD algorithm converges to a stationary point.   We first to prove two lemmas.

\begin{Lem} \label{l4.6}
Suppose that $\left\{ {\bf Z}^{(k)} \equiv \left(A^{(k)}, B^{(k)}, X^{(k)} \right) : k = 0, 1, 2, \cdots \right\}$ is the sequence generated by the LRQD algorithm and it is bounded.   Then there is a positive constant $\eta$ such that for $k = 1, 2, 3, \cdots$,
$$\left\|\left(A^{(k)}, B^{(k)}, X^{(k)}\right) - \left(A^{(k+1)}, B^{(k+1)}, X^{(k+1)}\right)\right\|_F \ge \eta \left\| \Pi_{\Sigma}(\nabla f(A^{(k)}, B^{(k)}, X^{(k)}))\right\|_F.$$
\end{Lem}
\begin{proof}  By (\ref{e8}) and Proposition \ref{p2.1}, we have
$$\begin{array} {rl} & \left\|\left(A^{(k)}B^{(k)}-X^{(k)}\right)\left(B^{(k)}\right)^*\right\|_F^2 \\
= & \left\|\left(A^{(k)}-A^{(k+1)}\right)\left[B^{(k)}\left(B^{(k)}\right)^*+ \lambda I_r\right]\right\|_F^2 \\
\le &  \left\|A^{(k)}-A^{(k+1)}\right\|_F^2 \left\|B^{(k)}\left(B^{(k)}\right)^*+ \lambda I_r\right\|_F^2 \\
\le & {1 \over \eta^2} \left\|A^{(k+1)} - A^{(k)}\right\|_F^2.\end{array}$$
Here $\eta$ is a positive constant.   Such a positive constant exists, as $\left\{ \left(A^{(k)}, B^{(k)}, X^{(k)} \right) : k = 0, 1, 2, \cdots \right\}$ is bounded.
Similarly, by (\ref{e9}) and Proposition \ref{p2.1}, we have
$$\left\|\left(A^{(k)}\right)^*\left(A^{(k)}B^{(k)}-X^{(k)}\right)\right\|_F^2 \le {1 \over \eta^2} \left\|B^{(k+1)} - B^{(k)}\right\|_F^2.$$
We also have
$$\left\|\left(X^{(k)}-A^{(k)}B^{(k)}\right)_{\Omega_C}\right\|_F^2 \le {1 \over \eta^2} \left\|X^{(k+1)} - X^{(k)}\right\|_F^2.$$

By (\ref{e17}), we have
$$\begin{array}{rl} & \left\| \Pi_{\Sigma}\left(\nabla f(A^{(k)}, B^{(k)}, X^{(k)})\right)\right\|_F^2 \\ & = \left\|\left(A^{(k)}B^{(k)}-X^{(k)}\right)\left(B^{(k)}\right)^*\right\|_F^2 + \left\|\left(A^{(k)}\right)^*\left(A^{(k)}B^{(k)}-X^{(k)}\right)\right\|_F^2 + \left\|\left(X^{(k)}-A^{(k)}B^{(k)}\right)_{\Omega_C}\right\|_F^2. \end{array}$$

Summing up the last three inequalities and using the last equality, we have the conclusion.
\end{proof}

\begin{Lem} \label{l4.7}
Let ${\bf Z}^{\#}$ be one limiting point of $\{ {\bf Z}^{(k)} \}$.   Assume that ${\bf Z}^{(0)}$ satisfies
${\bf Z}^{(0)} \in N$ and $\left\|{\bf Z}^{(0)} - {\bf Z}^{\#}\right\| \le \rho$, where
\begin{equation} \label{e18}
\rho > {2\mu \over \lambda \eta (1-\theta)}\left|f({\bf Z}^{(0)}) - f({\bf Z}^{\#})\right|^{1-\theta} + \left\|{\bf Z}^{(0)}-{\bf Z}^{\#}\right\|.
\end{equation}
Then we have the following conclusions:
\begin{equation} \label{e19}
\left\|{\bf Z}^{(k)} - {\bf Z}^{\#}\right\| \le \rho, \ \ {\rm for}\ k = 0, 1, 2, \cdots
\end{equation}
and
\begin{equation} \label{e20}
\sum_{k=0}^\infty \left\| {\bf Z}^{(k)} - {\bf Z}^{(k+1)} \right\| \le {2\mu \over \lambda \eta (1-\theta)}\left|f({\bf Z}^{(0)} - f({\bf Z}^{\#})\right|^{1-\theta}.
\end{equation}
\end{Lem}
\begin{proof} We prove (\ref{e19}) by induction.    By assumption, (\ref{e19}) holds for $k=0$.   We now assume that there is an integer $\hat k$ such that (\ref{e19}) holds for $0 \le k \le {\hat k}$.  This means that the {K}urdyka-{\L}ojasiewicz inequality holds at these points.   We now prove that (\ref{e19}) holds for $k= {\hat k}+1$.   Define a scalar function
$$\psi(\alpha):= {1 \over 1-\theta}\left|\alpha -f ({\bf Z}^{\#})\right|^{1-\theta}.$$
Then $\psi$ is a concave function and $\psi'(\alpha) = |\alpha - f({\bf Z}^{\#})|^{-\theta}$ if $\alpha \ge f({\bf Z}^{\#})$.  For $0 \le k \le \hat k$, we have
$$\begin{array}{rl}
& \psi\left(f\left({\bf Z}^{(k)}\right)\right) - \psi\left(f\left({\bf Z}^{(k+1)}\right)\right) \\
\ge & \psi'\left(f\left({\bf Z}^{(k)}\right)\right)\left[f\left({\bf Z}^{(k)}\right)-f\left({\bf Z}^{(k+1)}\right)\right] \\
\ge & {1 \over \left|f\left({\bf Z}^{(k)}\right)-f\left({\bf Z}^{(\#)}\right)\right|^\theta}{\lambda \over 2}\left\|{\bf Z}^{(k)} - {\bf Z}^{(k+1)}\right\|_F^2 \ \ \ {{\rm [by \ Lemma \ \ref{l4.3} ]}} \\
\ge & {\lambda \over 2\mu}{\left\|{\bf Z}^{(k)} - {\bf Z}^{(k+1)}\right\|_F^2 \over \left\|\Pi_{\Sigma}\left(\nabla f\left({\bf Z}^{(k)}\right)\right)\right\|} \ \ \ {{\rm [by \ (\ref{e16}) ]}} \\
\ge & {\lambda \eta \over 2\mu}{\left\|{\bf Z}^{(k)} - {\bf Z}^{(k+1)}\right\|_F^2 \over \left\|{\bf Z}^{(k)} - {\bf Z}^{(k+1)}\right\|_F} \ \ \ {{\rm [by \ Lemma \ \ref{l4.6} ]}} \\
\ge & {\lambda \eta \over 2\mu}\left\|{\bf Z}^{(k)} - {\bf Z}^{(k+1)}\right\|_F.
\end{array}$$
Summarizing $k$ from $0$ to $\hat k$, we have
\begin{eqnarray}
& & \sum_{k=0}^{\hat k} \left\|{\bf Z}^{(k)} - {\bf Z}^{(k+1)}\right\|_F \nonumber \\
& \le & {2\mu \over \lambda \eta}  \sum_{k=0}^{\hat k} \left[\psi\left(f\left({\bf Z}^{(k)}\right)\right) - \psi\left(f\left({\bf Z}^{(k+1)}\right)\right)\right] \nonumber \\
& = & {2\mu \over \lambda \eta} \left[\psi\left(f\left({\bf Z}^{(0)}\right)\right) - \psi\left(f\left({\bf Z}^{(\hat k+1)}\right)\right)\right] \nonumber \\
& \le & {2\mu \over \lambda \eta} \psi\left(f\left({\bf Z}^{(0)}\right)\right). \label{e21}
\end{eqnarray}
By this and (\ref{e18}), we have
$$\begin{array}{rl}
& \left\|{\bf Z}^{(\hat k+1)} - {\bf Z}^{(\#)}\right\|_F \\
\le & \sum_{k=0}^{\hat k} \left\|{\bf Z}^{(k)} - {\bf Z}^{(k+1)}\right\|_F + \left\|{\bf Z}^{(0)}-{\bf Z}^{\#}\right\| \\
\le & {2\mu \over \lambda \eta} \psi\left(f\left({\bf Z}^{(0)}\right)\right) + \left\|{\bf Z}^{(0)}-{\bf Z}^{\#}\right\| \\
< & \rho.
\end{array}$$
This proves (\ref{e19}).

Letting $\hat k \to \infty$ in (\ref{e21}) and using (\ref{e18}), we have (\ref{e20}).
\end{proof}

\begin{Thm} \label{t4.8}
Suppose that the LRQD algorithm generates a bounded sequence $\{ {\rm Z}^{(k)} \}$.  Then
$$\sum_{k=0}^\infty \left\| {\bf Z}^{(k)} - {\bf Z}^{(k+1)} \right\| \le \infty,$$
which means that the entire sequence $\{ {\rm Z}^{(k)} \}$ converges to a limit.
\end{Thm}
\begin{proof}
Because $\{ {\rm Z}^{(k)} \}$ is bounded, it must have a limiting point ${\bf Z}^{\#}$.  Then there is an index $k_0$ such that
$$\left\| {\bf Z}^{k_0} - {\bf Z}^{\#} \right\| \le \rho.$$
We may regard ${\bf Z}^{k_0}$ as the initial point.   Then Lemma \ref{l4.7} holds.  The entire sequence satisfies (\ref{e20}).   The theorem is proved.
\end{proof}

Finally, we establish convergence rates for the convergence of this sequence.

\begin{Thm}
Suppose that the LRQD algorithm generates a bounded sequence $\{ {\rm Z}^{(k)} \}$.    In (\ref{e16}),

(1) if
$\theta = {1 \over 2}$, then there exist $\sigma \in [0, 1)$ and $\beta > 0$ such that
$$\left\|{\bf Z}^{(k)} - {\bf Z}^{\#}\right\| \le \beta \sigma^k,$$
i.e., the sequence converges R-linearly;

(2) if ${1 \over 2} < \theta < 1$, then there exists $\beta > 0$ such that
$$\left\|{\bf Z}^{(k)} - {\bf Z}^{\#}\right\| \le \beta k^{-{1-\theta \over 2\theta - 1}}.$$
\end{Thm}
\begin{proof}
Without loss of generality, assume that $\left\| {\bf Z}^{(0)} - {\bf Z}^{\#}\right\| < \rho$.  Let
\begin{equation} \label{e22}
\Delta_k := \sum_{i=k}^\infty \left\| {\bf Z}^{(i)} - {\bf Z}^{(i+1)} \right\| \ge \left\| {\bf Z}^{(k)} - {\bf Z}^{\#} \right\|.
\end{equation}
By Lemma \ref{l4.7}, we have
\begin{eqnarray}
\Delta_k  & \le  & {2\mu \over \lambda \eta (1-\theta)}\left|f({\bf Z}^{(k)} - f({\bf Z}^{\#})\right|^{1-\theta} \nonumber \\
& = & {2\mu \over \lambda \eta (1-\theta)}\left[\left|f({\bf Z}^{(k)} - f({\bf Z}^{\#})\right|^\theta \right]^{1-\theta \over \theta} \nonumber \\
& \le & {2\mu \over \lambda \eta (1-\theta)} \mu^{1-\theta \over \theta} \left\|\Pi_{\Sigma}\left(\nabla f\left({\bf Z}^{(k)}\right)\right)\right\|^{1-\theta \over \theta}  \ \ \ {{\rm [by \ (\ref{e16}) ]}} \nonumber \\
& \le & {2\mu \over \lambda \eta (1-\theta)} \left({\mu \over \eta}\right)^{1-\theta \over \theta} \left\|{\bf Z}^{(k)} - {\bf Z}^{(k+1)} \right\|^{1-\theta \over \theta}  \ \ \ {{\rm [by \ Lemma \ \ref{l4.6} ]}} \nonumber \\
& = & {2\mu^{1 \over \theta} \over \lambda \eta^{1 \over \theta} (1-\theta)}  \left\|{\bf Z}^{(k)} - {\bf Z}^{(k+1)} \right\|^{1-\theta \over \theta}.   \label{e23}
\end{eqnarray}

(1) If $\theta = {1 \over 2}$, then ${1 - \theta \over \theta} = 1$.  By (\ref{e23}), we have
$$\Delta_k \le {2\mu^{1 \over \theta} \over \lambda \eta^{1 \over \theta} (1-\theta)} \left(\Delta_k - \Delta_{k+1}\right).$$
This implies that
\begin{equation} \label{e24}
\Delta_{k+1} \le \sigma \Delta_k,
\end{equation}
where
$$\sigma = {2\mu^{1 \over \theta} - \lambda \eta^{1 \over \theta} (1-\theta) \over 2\mu^{1 \over \theta}}.$$
By (\ref{e22}) and (\ref{e23}), we know that
$$\left\|{\bf Z}^{(k)} - {\bf Z}^{\#}\right\| \le \Delta_k \le \sigma \Delta_{k-1} \le \cdots \le \sigma^k \Delta_0 \equiv \beta \sigma^k,$$
where $\beta \equiv \Delta_0$ is finite by Theorem \ref{t4.8}.   The conclusion follows.

(2)  Let
$$\kappa^{1-\theta \over \theta} = {2\mu^{1 \over \theta} \over \lambda \eta^{1 \over \theta} (1-\theta)}.$$
By (\ref{e23}),
$$\Delta_k^{\theta \over 1-\theta} \le \kappa \left(\Delta_k - \Delta_{k+1}\right).$$
Define $\zeta(\alpha) := \alpha^{-{\theta \over 1-\theta}}$.   Then $\zeta$ is monotonically decreasing.  We have
$$\begin{array}{rl}
{1 \over \kappa} \le & \zeta(\Delta_k)\left(\Delta_k - \Delta_{k+1}\right) \\
= & \int_{\Delta_{k+1}}^{\Delta_k}\zeta(\Delta_k) d\alpha \\
\le & \int_{\Delta_{k+1}}^{\Delta_k}\zeta(\alpha) d\alpha \\
= & -{1-\theta \over 2\theta -1}\left(\Delta_k^{-{2\theta-1 \over 1-\theta}} - \Delta_{k+1}^{-{2\theta-1 \over 1-\theta}} \right).
\end{array}$$
Denote $\nu := {-{2\theta-1 \over 1-\theta}}$.  Then $\nu < 0$ since ${1 \over 2} < \theta < 1$.
We have
$$\Delta_{k+1}^\nu - \Delta_{k}^\nu \ge \xi > 0,$$
where $\xi \equiv -{\nu \over \kappa}$.  This implies that
$$\Delta_k \le \left[\Delta_0^\nu + k\xi\right]^{1 \over \nu} \le (k\xi)^{1 \over \nu}.$$
Letting $\beta = \xi^{1 \over \nu}$, we have the conclusion.
\end{proof}

%\section{Numerical Experiments}

\section{Concluding Remarks}

In a certain sense, the gradient of a real function in quaternion matrix variables, studied in Section 4, is a kind of the pseudo-derivative mentioned in \cite{XJTM15}.   What we presented in (\ref{n14}) is a concise form of such a pseudo-derivative.   It turns out to give us simple expressions for gradients of functions, and optimality conditions of optimization problems, involved in our convergence analysis.  It may be worth further exploring the use of (\ref{n14}).

\bigskip

{\bf Acknowledgment}  We are thankful to Prof. Qingwen Wang for his help on quaternion matrices.

%{\bf Data availability statement}    The datasets generated during and/or analysed during the current study are available from the corresponding author on reasonable request.

% \vspace{100pt}

\end{document}